\documentclass[preprint,12pt,3p]{elsarticle}
\usepackage{amsmath,amsthm,amsfonts,amssymb}
\usepackage{hyperref}
\usepackage{yhmath}
\usepackage{MnSymbol}  % used for  udots
\usepackage[normalem]{ulem}
\usepackage{sectsty}

\usepackage{colortbl}
\usepackage[dvipsnames]{xcolor}

\newtheorem{thm}{Theorem}[section]

\newtheorem{cor}[thm]{Corollary}
\newtheorem{lem}[thm]{Lemma}

\numberwithin{equation}{section}

\definecolor{darkslategray}{rgb}{0.18, 0.31, 0.31}
\definecolor{warmblack}{rgb}{0.0, 0.26, 0.26}
\definecolor{astral}{RGB}{46,116,181}
\subsectionfont{\color{astral}}
\sectionfont{\color{astral}}
\journal{....................... }
\begin{document}
	\begin{frontmatter}
		\title{ \textcolor{warmblack}{\bf Further results on $\mathbb{A}$-numerical radius inequalities  }}

	\author[label1]{Nirmal Chandra Rout}\ead{nrout89@gmail.com}
		
		\author[label1]{Debasisha Mishra\corref{cor2}}\ead{dmishra@nitrr.ac.in}
		
		\address[label1]{Department of Mathematics, National Institute of Technology Raipur, Raipur-492010, India}
		\cortext[cor2]{Corresponding author}

		\begin{abstract}
			\textcolor{warmblack}{
			Let  $A$ be a  bounded  linear positive operator on a complex Hilbert space $\mathcal{H}.$ Further, let $\mathcal{B}_A\mathcal{(H)}$ denote the set of all bounded linear operators on $\mathcal{H}$ whose $A$-adjoint exists, and
            $\mathbb{A}$ signify a diagonal  operator matrix with diagonal entries are $A.$ 
Very recently, several $A$-numerical radius inequalities of $2\times 2 $ operator matrices were established by Feki and Sahoo [arXiv:2006.09312; 2020]  and  Bhunia {\it et al.} [Linear Multilinear Algebra (2020), DOI: 10.1080/03081087.2020.1781037],  assuming the conditions ``$\mathcal{N}(A)^\perp$ is  invariant under  different operators in $\mathcal{B}_A(\mathcal{H})$'' and ``$A$ is strictly positive'', respectively. 
             In this paper, we prove a few new $\mathbb{A}$-numerical radius inequalities for  $2\times 2$ and $n\times n$ operator matrices. We also provide some new proofs of the existing results by relaxing different sufficient conditions like  ``$\mathcal{N}(A)^\perp$ is invariant under different operators'' and ``$A$ is strictly positive''. Our proofs show the importance of the theory of the Moore-Penrose inverse of a bounded linear  operator in this field of study.     			  }
		\end{abstract}
		
 	\end{frontmatter}

	\section{Introduction }\label{intro}
Throughout  $\mathcal{H}$ denotes  a complex Hilbert space with inner product $\langle \cdot,\cdot\rangle$. By $\mathcal{B}(\mathcal{H})$, we mean  the $C^*$-algebra of all   bounded linear operators on $\mathcal{H}$.
	Let $\|\cdot\|$ be the norm induced from $\langle \cdot,\cdot\rangle.$
For $A\in\mathcal{B}(\mathcal{H})$,	$\mathcal{R}(A)$ stands for the range space of $A$ and $\overline{\mathcal{R}(A)}$ for the norm closure of $\mathcal{R}(A)$ in $\mathcal{H}$. And  $A^{*}$ represents the adjoint of $A$.
An operator $A\in\mathcal{B}(\mathcal{H})$ is called {\it selfadjoint} if $A=A^{*}$. A selfadjoint operator $A\in\mathcal{B}(\mathcal{H})$ is called {\it positive} if $\langle Ax, x\rangle \geq 0$ for all $x \in \mathcal{H}$, and is called {\it strictly positive} if  $\langle Ax, x\rangle > 0$ for all non-zero $x\in \mathcal{H}$. If $A$ is a positive (strictly positive) operator,  then we use the notation $A \geq 0$ ($A > 0$).    Let $\mathbb{A}$ be an $n\times n$ diagonal operator matrix whose diagonal entries are positive operator $A$ for $n=1,2,...$. Then $\mathbb{A} \in\mathcal{B}(\bigoplus_{i=1}^n \mathcal{H})$ and $\mathbb{A}\geq 0$. If $A\geq 0$, then it induces a positive semidefinite sesquilinear form, $\langle \cdot,\cdot\rangle_A: \mathcal{H}\times\mathcal{H}\rightarrow\mathbb{C}$ defined by $\langle x, y \rangle_A=\langle Ax,y\rangle,$ $x,y\in\mathcal{H}.$ 
 Let $\|\cdot\|_A$ denote the seminorm on $\mathcal{H}$ induced by $\langle \cdot, \cdot \rangle_A,$ i.e., $\|x\|_A=\sqrt{\langle x, x \rangle_A}$ for all $x \in \mathcal{H}.$ Then $\|x\|_A$ is a norm
if, and only if, $A>0$. Also, $(\mathcal{H}, \|\cdot\|_A)$ is complete if, and only if,  $\mathcal{R}(A)$ is closed in $\mathcal{H}.$ Henceforth, we use the symbol $A$ and $\mathbb{A}$ for positive operators on $\mathcal{H}$ and $\bigoplus_{i=1}^n \mathcal{H}$, respectively. We  retain the notation $O$ and $I$ for the null operator  and the identity operator on $\mathcal{H}$, respectively. Given $T \in\mathcal{B}\mathcal{(H)}$, 
 the $A$-operator seminorm  $\|T\|_A$ 
 is defined as follows:
$$\|T\|_A=\sup_{x\in \overline{\mathcal{R}(A)},~x\neq 0}\frac{\|Tx\|_A}{\|x\|_A}=\inf\left\{c>0: \|Tx\|_A\leq c\|x\|_A, 0\neq x\in \overline{\mathcal{R}(A)}\right\}<\infty.$$
We set $\mathcal{B}^A\mathcal{(H)}=\{T\in \mathcal{B(H)}:\|T\|_A<\infty\}.$ Then $\mathcal{B}^A\mathcal{(H)}$ is not a subalgebra of $\mathcal{B(H)}$.  It is pertinent to point out that $\|T\|_A=0$ if, and only if, $ATA=O.$ For $T\in\mathcal{B}^A\mathcal{(H)},$ we have 
$$\|T\|_A=\sup \{|\langle Tx,y\rangle_A|: x,y\in \overline{\mathcal{R}(A),} ~\|x\|_A=\|y\|_A=1\}.$$
If $AT\geq 0$, then the operator $T$ is called {\it $A$-positive}.  Note that if $T$ is $A$-positive, then 
$$\|T\|_A=\sup \{\langle Tx,x\rangle_A: x\in \mathcal{H},  \|x\|_A=1\}.$$
Before we proceed further, it is necessary to introduce the concept of $A$-adjoint operator.  We say an operator $X\in \mathcal{B(H)}$ to be {\it $A$-adjoint operator} of $T\in \mathcal{B(H)}$ if  $\langle Tx,y\rangle_A=\langle x, Xy\rangle_A$  for every $x,y\in \mathcal{H},$ i.e., $AX=T^*A.$
By Douglas Theorem \cite{Doug}, the existence of an $A$-adjoint operator is not guaranteed.  An operator $T\in \mathcal{B(H)}$ may admit none, one or many $A$-adjoints.  A rather well known result states that $A$-adjoint  of an operator $T\in \mathcal{B}\mathcal{(H)}$ exists if, and only if, $\mathcal{R}(T^*A)\subseteq \mathcal{R}(A)$. Let us now denote  $\mathcal{B}_A\mathcal{(H)}=\{T\in \mathcal{B(H)}:\mathcal{R}(T^*A)\subseteq \mathcal{R}(A)\}.$ Note that $\mathcal{B}_A\mathcal{(H)}$ is a subalgebra of $\mathcal{B(H)}$ which is neither closed nor dense in $\mathcal{B(H)}.$ Moreover, we have the following inclusion relations: $$\mathcal{B}_A\mathcal{(H)}\subseteq \mathcal{B}^A\mathcal{(H)}\subseteq\mathcal{B}\mathcal{(H)}. $$ And the equality holds if, $A$ is injective and has a closed range.

For  $T\in \mathcal{B(H)}$, $w_A(T)$, the {\it $A$-numerical radius} of $T$ was proposed by Saddi \cite{Saddi}. And is defined as follows:
\begin{equation}\label{eqn1a}
    w_A(T)=\sup\{|\langle Tx,x\rangle_A|:x\in \mathcal{H}, \|x\|_A=1\}.
\end{equation}
Very recently, Zamani \cite{Zam} obtained the following  $A$-numerical radius inequality for $T\in \mathcal{B}_A(\mathcal{H})$:
\begin{align}\label{ineq1}
    \frac{1}{2}\|T\|_A\leq w_A(T)\leq \|T\|_A.
\end{align}
The first inequality in \eqref{ineq1} becomes an equality if $T^2= O$ and the second inequality becomes an equality if $T$ is $A$-selfadjoint.
The $A$-Crawford number of $T\in\mathcal{B}_A(\mathcal{H})$ is  defined  as
$$c_A(T)=\inf\{|\langle Tx,x\rangle_A|:x\in\mathcal{H},\|x\|_A=1\}.$$
This terminology was introduced by Zamani \cite{Zam}.
Furthermore,  if $T$ is $A$-selfadjoint, then $w_A(T)=\|T\|_A$.
    Moslehian {\it et al.} \cite{MOS}  continued the study of $A$-numerical radius and obtained some new $A$-numerical radius inequalities.
   In this year, Bhunia {\it et al.}  \cite{PINTU, PINTU1} presented several $\mathbb{A}$-numerical radius inequalities for a strictly positive operator $A$. Feki  \cite{Feki_some}, and Feki and Sahoo \cite{FekSat} established some more $A$-numerical radius inequalities under the assumption ``$\mathcal{N}(A)^\perp$ is invariant under  different operators''. We refer the interested reader to \cite{ Feki, NSD} and the references cited therein for 
   further generalizations and refinements of $A$-numerical radius inequalities.
%   \DM{$A$-numerical radius or $A$-numerical radius inequalities Check what type of work the authors did there}
   
The objective of this paper is 
to present a few new $\mathbb{A}$-numerical radius inequalities for  $2\times 2$ and $n\times n$ operator matrices. Besides these, we aim to establish some existing $\mathbb{A}$-numerical radius inequalities by relaxing sufficient conditions like  $A>0$ and $\mathcal{N}(A)^\perp$ is invariant under different operators in $\mathcal{B}_A(\mathcal{H}).$
To this end, the paper is sectioned as follows.  In Section 2, we define additional mathematical constructs including
the definition of the Moore-Penrose inverse of an operator, $A$-adjoint, $A$-selfadjoint and $A$-unitary operator, that are required to state and prove the results in the subsequent sections.
Section 3 contains several new $A$-numerical radius inequalities. More interestingly,  it also provides new proof to the very recent existing results in the literature on $A$-numerical radius inequalities by dropping some sufficient conditions. 
     
 \section{Preliminaries}
This section gathers a few more definitions and results that are useful in proving our main results.  It starts with the definition of the Moore-Penrose inverse of a bounded operator $A$ in $H$. 
The {\it Moore-Penrose inverse} of $A\in \mathcal{B(H)}$ \cite{Nashed} is the
operator $X:R(A)\bigoplus R(A)^{\perp} \longrightarrow \mathcal{H}$ which satisfies the following four equations:
\begin{center}
(1) $AXA = A$,~ (2) $XAX = X$,~ (3) $X A= P_{N(A)^{\perp}}$,~~(4) $A X=P_{\overline{\mathcal{R}(A)}}|_{R(A)\bigoplus R(A)^{\perp}}$.
\end{center}
Here $N(A)$  and  $P_L$ denote the null space of $A$ and the orthogonal projection onto $L$, respectively.
The Moore-Penrose inverse
is unique, and  is denoted by $A^\dagger.$ In general, $A^{\dagger}\notin \mathcal{B(H)}$. It is bounded if and only if  $\mathcal{R}(A)$  is closed. 
 If $A\in \mathcal{B(H)}$ is invertible, then $A^\dagger= A^{-1}.$ 
If $T\in \mathcal{B}_A\mathcal{(H)},$ the reduced solution of the equation $AX=T^*A$ is a distinguished $A$-adjoint operator of $T,$ which is denoted by $T^{\#_A}$ (see \cite{ARIS2,Mos}). Note that $T^{\#_A}=A^\dagger T^* A$. If $T\in \mathcal{B}_A(\mathcal{H}),$ then $AT^{\#_A}=T^*A$, $\mathcal{R}(T^{\#_A})\subseteq \overline{\mathcal{R}(A)}$ and $\mathcal{N}(T^{\#_A})=\mathcal{N}(T^*A)$ (see \cite{Doug}). 
One can observe that
\begin{align}\label{eqn00_1.7}
    I^{\#_A}=A^{\dagger}I^*A=A^{\dagger}A=P_{\overline{\mathcal{R}(A)}}~~(\because ~\mathcal{N}(A)^\perp=\mathcal{R}(A^*)). 
\end{align}
Besides, we derive below two new properties of $A$-adjoint of an operator $T\in \mathcal{B}_A(\mathcal{H}),$ which are crucial in providing some new proofs of the existing results and in proving new results on $A$-numerical radius inequalities. 
\begin{equation}\label{eqn00_1.8}
    T^{\#_A}P_{\overline{\mathcal{R}(A)}}=A^\dagger T^*AA^\dagger A=A^\dagger T^*A=T^{\#_A},
\end{equation}
and
\begin{equation}\label{eqn00_1.9}
    P_{\overline{\mathcal{R}(A)}}T^{\#_A}=A^\dagger A A^\dagger T^*A=A^\dagger T^*A=T^{\#_A}.
\end{equation}
An operator $T\in \mathcal{B(H)}$ is said to be {\it $A$-selfadjoint} if $AT$ is selfadjoint, i.e., $AT=T^*A.$ Observe that if $T$ is $A$-selfadjoint, then $T\in \mathcal{B}_A(\mathcal{H}).$  However,  in general, $T\neq T^{\#_A}.$ But, $T=T^{\#_A}$  if and only if $T$ is $A$-selfadjoint and $\mathcal{R}(T)\subseteq \overline{\mathcal{R}(A)}.$ If $T\in \mathcal{B}_A(\mathcal{H}),$ then $T^{\#_A}\in \mathcal{B}_A(\mathcal{H}),$  $(T^{\#_A})^{\#_A}=P_{\overline{\mathcal{R}(A)}}TP_{\overline{\mathcal{R}(A)}},$  and $\left((T^{\#_A})^{\#_A}\right)^{\#_A}=T^{\#_A}.$ Also, $T^{\#_A}T$ and $TT^{\#_A}$ are  $A$-positive operators, and
\begin{align}\label{ineq0}
    \|T^{\#_A}T\|_A=\|TT^{\#_A}\|_A=\|T\|_A^2=\|T^{\#_A}\|_A^2.
\end{align}
For any $T_1,T_2\in \mathcal{B}_A(\mathcal{H}),$ we have
\begin{align}
    \|T_1^{\#_A}T_2\|_A&=\sup\{|\langle T_1^{\#_A}T_2 x,y\rangle|:x,y\in\mathcal{H}, ~\|x\|_A=\|y\|_A=1\}\nonumber\\
    &=\sup\{|\langle T_2 x,T_1y\rangle|:x,y\in\mathcal{H}, ~\|x\|_A=\|y\|_A=1\}\nonumber\\
    &=\sup\{|\langle  x,T_2^{\#_A}T_1y\rangle|:x,y\in\mathcal{H}, ~\|x\|_A=\|y\|_A=1\}\nonumber\\
    &=\sup\{|\langle  T_2^{\#_A}T_1 y,x\rangle|:x,y\in\mathcal{H}, ~\|x\|_A=\|y\|_A=1\}\nonumber\\
    &=\|T_2^{\#_A}T_1\|_A.
\end{align}
This fact is  same as Lemma 2.8 of \cite{FekSat}. However, the above proof is a very simple one and directly follows using the definition of $A$-norm. An operator $U\in \mathcal{B}_A(\mathcal{H})$ is said to be {\it $A$-unitary} if $\|Ux\|_A=\|U^{\#_A}x\|_A=\|x\|_A$ for all $x\in \mathcal{H}.$ If $T\in \mathcal{B}_A(\mathcal{H})$ and $U$ is $A$-unitary, then $w_A(U^{\#_A}TU)=w_A(T).$
For $T,S\in \mathcal{B}_A(\mathcal{H}),$ we have $(TS)^{\#_A}=S^{\#_A}T^{\#_A},$  $(T+S)^{\#_A}=T^{\#_A}+S^{\#_A},$ $\|TS\|_A\leq \|T\|_A\|S\|_A$ and $\|Tx\|_A\leq \|T\|_A\|x\|_A$ for all $x\in \mathcal{H}.$  The real and imaginary part of an operator $T\in \mathcal{B}_A(\mathcal{H})$  as $Re_A(T)=\frac{T+T^{\#_A}}{2}$ and $Im_A(T)=\frac{T-T^{\#_A}}{2i}$.  
An interested reader may refer \cite{ARIS,ARIS2} for further properties of operators on Semi-Hilbertian space.
From the definition of $A$-numerical radius of an operator, it follows that
\begin{equation}\label{eqn_00000001.4}
   w_A(T)=w_A(T^{\#_A}) \mbox{  for   any  }   T\in\mathcal{B}_A(\mathcal{H}). 
\end{equation}
Some interesting results are collected hereunder for further use.
\begin{lem}\label{lem_fek}\textnormal{(Lemma 3.1, \cite{pinfek})}\\
Let $T_{ij}\in \mathcal{B}_A(\mathcal{H})$ for  $1\leq i,j\leq n.$ Then $$T=\begin{bmatrix}
 T_{11} & T_{12}  & \cdots  &   T_{1n}\\
 T_{21}  & T_{22} & \cdots  & T_{2n}\\
 \vdots & \vdots &\vdots & \vdots \\
T_{n1}  & T_{n2}  &\cdots&  T_{nn}
\end{bmatrix}\in \mathcal{B}_{\mathbb{A}}(\mathcal{H}) \mbox{ and }
T^{\#_{\mathbb{A}}}=\begin{bmatrix}
 T_{11}^{\#_A} & T_{21}^{\#_A}  & \cdots  &   T_{n1}^{\#_A}\\
 T_{12}^{\#_A}  & T_{22}^{\#_A} & \cdots  & T_{n2}^{\#_A}\\
 \vdots & \vdots &\vdots & \vdots \\
T_{1n}^{\#_A}  & T_{2n}^{\#_A}  &\cdots&  T_{nn}^{\#_A}
\end{bmatrix}.$$
\end{lem}
The next result is a combination of Lemma 2.4 (i) \cite{PINTU} and Lemma 2.2 \cite{Nirmal2}.
\begin{lem}\label{l001}
 Let  $T_1, T_2, T_3, T_4\in \mathcal{B}_A(\mathcal{H}).$  Then
  \begin{enumerate}
        \item [\textnormal{(i)}] $\max\{w_A(T_1), w_A(T_4)\}= w_\mathbb{A}\left(\begin{bmatrix}
     T_1 & O\\
     O & T_4
     \end{bmatrix}\right)\leq w_\mathbb{A}\left(\begin{bmatrix}
     T_1 & T_2\\
     T_3 & T_4
     \end{bmatrix}\right).$
      \item [\textnormal{(ii)}] $w_\mathbb{A}\left(\begin{bmatrix}
     O & T_2\\
     T_3 & O
     \end{bmatrix}\right)\leq w_\mathbb{A}\left(\begin{bmatrix}
     T_1 & T_2\\
     T_3 & T_4
     \end{bmatrix}\right).$
        \end{enumerate}
 \end{lem}
 
The other parts of Lemma 2.4 \cite{PINTU} assume the condition $A$ is strictly positive. Rout {\it et al.} \cite{Nirmal2} proved the same result for positive $A$, and the same is stated below. 
 
\begin{lem}\label{lemm_0000}[Lemma 2.4, \cite{Nirmal2}]\\
  Let  $T_1, T_2 \in \mathcal{B}_A(\mathcal{H}).$  Then
        \begin{enumerate}
    \item [\textnormal{(i)}] $w_\mathbb{A}\left(\begin{bmatrix}
    O & T_1\\
     T_2 & O
    \end{bmatrix}\right)=w_\mathbb{A}\left(\begin{bmatrix}
    O & T_2\\
     T_1 & O
    \end{bmatrix}\right).$\\
    \item [\textnormal{(ii)}] $w_\mathbb{A}\left(\begin{bmatrix}
    O & T_1\\
     e^{i\theta}T_2 & O
    \end{bmatrix}\right)=w_\mathbb{A}\left(\begin{bmatrix}
    O & T_1\\
     T_2 & O
    \end{bmatrix}\right)$ for~any~$\theta\in\mathbb{R}$.\\
    \item [\textnormal{(iii)}]  $w_\mathbb{A}\left(\begin{bmatrix}
    T_1 & T_2\\
     T_2 & T_1
    \end{bmatrix}\right)=\max\{w_A(T_1+T_2),w_A(T_1-T_2)\}.$
     In particular, $w_\mathbb{A}\left(\begin{bmatrix}
    O & T_2\\
     T_2 & O
    \end{bmatrix}\right)=w_A(T_2).$
\end{enumerate}
\end{lem}
The next result establishes upper and lower bounds for the $\mathbb{A}$-numerical radius of a particular type of $2\times 2$ operator matrix that is a generalization of \eqref{ineq1}.
\begin{lem}\label{thm_002.6}[Theorem 2.6, \cite{Nirmal2}]\\
 Let $T_1, T_2\in \mathcal{B}_A(\mathcal{H}).$ Then 
\begin{equation}\label{eq01}
    \max\{w_A(T_1), w_A(T_2)\}\leq w_{\mathbb{A}}\left(\begin{bmatrix}
T_1 & T_2\\
-T_2 & -T_1
\end{bmatrix}\right)\leq w_A(T_1)+w_A(T_2).
\end{equation}
\end{lem}

\begin{lem}\label{l002}[Lemma 2.8, \cite{Nirmal2}]\\
Let $T_1, T_2 \in \mathcal{B}_A(\mathcal{H}).$ Then
$$w_{\mathbb{A}}\left(\begin{bmatrix}
T_2 & -T_1\\
T_1 & T_2
\end{bmatrix}\right)=\max\{w_A(T_1+iT_2), w_A(T_1-iT_2)\} .$$
\end{lem}

\begin{lem}\label{lemm00001}[Theorem 2.6, \cite{Feki_some}]\\
Let $T,S\in\mathcal{B}_A({\mathcal{H}}).$ Then $$w_A(TS\pm ST^{\#_A})\leq 2\|T\|_Aw_A(S).$$
\end{lem}

\section{Main Results}
  We begin this section with the following result which provides an estimate for $A$-operator norms of certain $2\times 2$ operator matrices.
% \ref{cor_3.7}. 

\begin{thm}\label{thm_00003.8}
Let $T\in\mathcal{B}_A(\mathcal{H})$ and $z_1,z_2\in \mathbb{C}$. Then
$$\left\|\begin{bmatrix}z_1I&T\\O&z_2I
\end{bmatrix}\right\|_\mathbb{A}=\frac{1}{\sqrt{2}}\sqrt{|z_1|^2+|z_2|^2+\|T\|_A^2+\sqrt{(|z_1|^2+|z_2|^2+\|T\|_A^2)^2-4|z_1|^2|z_2|^2}}.
$$
\end{thm}

\begin{proof}
Let $\alpha,\beta \in \mathbb{R}$ such that $\alpha^2+\beta^2=1$ and 
\begin{align}\label{eqn_al1}
    \left\|\begin{bmatrix}|z_1|& \|T\|_A\\O&|z_2|
\end{bmatrix}\right\|&=\left\|\begin{bmatrix}|z_1|& \|T\|_A\\O&|z_2|
\end{bmatrix}\begin{bmatrix}\alpha\\\beta
\end{bmatrix}\right\|\nonumber\\
&=\left\|\begin{bmatrix}|z_1|\alpha+\|T\|_A\beta\\|z_2|\beta
\end{bmatrix}\right\|\nonumber\\
&=\sqrt{|z_2|^2\beta^2+(|z_1|\alpha+\|T\|_A\beta)^2}.
\end{align} 
Let $x_n,y_n\in\mathcal{H}$ be two unit vectors in $\mathcal{H}$ such that $\displaystyle\lim_{n\rightarrow\infty}|\langle Ty_n,x_n\rangle|=\|T\|_A$ for $n\in\mathbb{N}.$ 
Let $a_n\in\mathbb{R}$ be such that $\overline{z_1}\langle Ty_n,x_n\rangle_A=e^{ia_n}|z_1|\langle Ty_n,x_n\rangle_A.$
Suppose that $\begin{bmatrix}\alpha e^{ia_n}x_n\\\beta y_n
\end{bmatrix}$ be a sequence in $\mathcal{H}\bigoplus\mathcal{H}.$ We can see that $\left\|\begin{bmatrix}\alpha e^{ia_n}x_n\\\beta y_n
\end{bmatrix}\right\|_\mathbb{A}=1.$ Now,
\begin{align}\label{eqnnn_003.8}
    \left\|\begin{bmatrix} z_1I & T\\O&z_2I
    \end{bmatrix}\right\|_\mathbb{A}&\geq\left\|\begin{bmatrix} 
z_1I & T\\O&z_2I
    \end{bmatrix}\begin{bmatrix}\alpha e^{ia_n}x_n\\\beta y_n
\end{bmatrix}\right\|_\mathbb{A}\nonumber\\
\end{align}
\begin{align}
&=\left\|\begin{bmatrix}\alpha z_1 e^{ia_n}x_n+\beta Ty_n\\\beta z_2 y_n
\end{bmatrix}\right\|_\mathbb{A}\nonumber\\
&=\sqrt{\|\alpha z_1e^{ia_n}x_n+\beta Ty_n\|_A^2+\|\beta z_2 y_n\|_A^2}\nonumber\\
&=\sqrt{\alpha^2|z_1|^2+\beta^2\|Ty_n\|_A^2+2\alpha\beta Re(\overline{z_1}\langle Ty_n,x_n\rangle_A)+\beta^2|z_2|^2}\nonumber\\
&=\sqrt{(\alpha|z_1|+\beta \|T\|_A)^2+\beta^2|z_2|^2}\nonumber\\
&=\left\|\begin{bmatrix}
|z_1|&\|T\|_A\\O&|z_2|
\end{bmatrix}\begin{bmatrix}
\alpha\\\beta
\end{bmatrix}\right\|~~~\mbox{by}~\eqref{eqn_al1}\nonumber\\
&=\left\|\begin{bmatrix}
|z_1|&\|T\|_A\\O&|z_2|
\end{bmatrix}\right\|.
\end{align}
Again, by Lemma 2.1 \cite{Feki1} 
\begin{equation}\label{eqnnn_003.9}
    \left\|\begin{bmatrix}
z_1I& T\\O&z_2I
\end{bmatrix}\right\|_\mathbb{A}\leq\left\|\begin{bmatrix}
|z_1|& \|T\|_A\\O&|z_2|
\end{bmatrix}\right\|.
\end{equation}
From \eqref{eqnnn_003.8} and \eqref{eqnnn_003.9}, we so have $$\left\|\begin{bmatrix}
z_1I& T\\O&z_2I
\end{bmatrix}\right\|_\mathbb{A}=\left\|\begin{bmatrix}
|z_1|& \|T\|_A\\O&|z_2|
\end{bmatrix}\right\|.$$
But 
\begin{align*}
    \left\|\begin{bmatrix}
|z_1|& \|T\|_A\\O&|z_2|
\end{bmatrix}\right\|&=r^{1/2}\left(\begin{bmatrix}
|z_1|& O\\\|T\|_A&|z_2|
\end{bmatrix}\begin{bmatrix}
|z_1|& \|T\|_A\\O&|z_2|
\end{bmatrix}\right) \\
&=r^{1/2}\left(\begin{bmatrix}
|z_1|^2& |z_1|\|T\|_A\\|z_1|\|T\|_A&|z_2|^2+\|T\|_A^2
\end{bmatrix}\right)\\
&=\frac{1}{\sqrt{2}}\sqrt{|z_1|^2+|z_2|^2+\|T\|_A^2+\sqrt{(|z_1|^2+|z_2|^2+\|T\|_A^2)^2-4|z_1|^2|z_2|^2}}.
\end{align*}
Hence, 
$$\left\|\begin{bmatrix}
z_1I& T\\O&z_2I
\end{bmatrix}\right\|_\mathbb{A}=\frac{1}{\sqrt{2}}\sqrt{|z_1|^2+|z_2|^2+\|T\|_A^2+\sqrt{(|z_1|^2+|z_2|^2+\|T\|_A^2)^2-4|z_1|^2|z_2|^2}}.$$
\end{proof}

We recall below a result of \cite{Feki_some} to obtain  Corollary \ref{cor_3.7}. 

\begin{lem}\label{lem0003.7}[Corollary 2.1, \cite{Feki_some}]\\
Let $T\in\mathcal{B}_A(\mathcal{H}).$ Then
$$\frac{1}{2}\sqrt{\|TT^{\#_A}+T^{\#_A}T\|_A+2c_A(T^2)}\leq w_A(T)\leq \frac{1}{2}\sqrt{\|TT^{\#_A}+T^{\#_A}T\|_A+2w_A(T^2)}.$$
\end{lem}

Next, we turn our attention towards a result by Feki \cite{Feki_some} that holds  with the additional assumption  ``$\mathcal{N}(A)^\perp$ is invariant under $T\in\mathcal{B}_A(\mathcal{H}).$''
We prove the same result without this assumption in the following corollary.

\begin{cor}\label{cor_3.7}
Let $T\in\mathcal{B}_A(\mathcal{H}).$ Then
$$2w_A\left(\begin{bmatrix}
I& T\\O&-I
\end{bmatrix}\right)=\left\|\begin{bmatrix}
I& T\\O&-I
\end{bmatrix}\right\|_\mathbb{A}+\left\|\begin{bmatrix}
I& T\\O&-I
\end{bmatrix}\right\|_\mathbb{A}^{-1}.$$
\end{cor}

\begin{proof}
Let $\mathbb{T}=\begin{bmatrix}
I& T\\O&-I
\end{bmatrix}.$
Then $\mathbb{T}^2=\begin{bmatrix}
I&O\\O&I
\end{bmatrix}.$ 
Using Lemma \ref{lem0003.7},  we get 
\begin{equation}\label{eqn0003.8}
    w_\mathbb{A}(\mathbb{T})=\frac{1}{2}\sqrt{\|\mathbb{T}\mathbb{T}^{\#_A}+\mathbb{T}^{\#_A}\mathbb{T}\|_\mathbb{A}+2}.
\end{equation}
From \eqref{eqn0003.8}, we now have 
\begin{align*}
    w_{\mathbb{A}}(\mathbb{T})&=\frac{1}{2}\sqrt{\left\|\begin{bmatrix}
I& T\\O&-I
\end{bmatrix}\begin{bmatrix}
I& T\\O&-I
\end{bmatrix}^{\#_\mathbb{A}}+\begin{bmatrix}
I& T\\O&-I
\end{bmatrix}^{\#_\mathbb{A}}\begin{bmatrix}
I& T\\O&-I
\end{bmatrix}\right\|_\mathbb{A}+2}\\
&=\frac{1}{2}\sqrt{\left\|\begin{bmatrix}
I& T\\O&-I
\end{bmatrix}\begin{bmatrix}
P_{\overline{\mathcal{R}(A)}}& O\\T^{\#_A}&-P_{\overline{\mathcal{R}(A)}}
\end{bmatrix}+\begin{bmatrix}
P_{\overline{\mathcal{R}(A)}}& O\\T^{\#_A}&-P_{\overline{\mathcal{R}(A)}}
\end{bmatrix}\begin{bmatrix}
I& T\\O&-I
\end{bmatrix}\right\|_\mathbb{A}+2}\\
&=\frac{1}{2}\sqrt{\left\|\begin{bmatrix}
P_{\overline{\mathcal{R}(A)}}+ TT^{\#_A}& -TP_{\overline{\mathcal{R}(A)}}\\-T^{\#_A}&P_{\overline{\mathcal{R}(A)}}
\end{bmatrix}+\begin{bmatrix}
P_{\overline{\mathcal{R}(A)}}& P_{\overline{\mathcal{R}(A)}}T \\T^{\#_A}&T^{\#_A}T+P_{\overline{\mathcal{R}(A)}}
\end{bmatrix}\right\|_\mathbb{A}+2}\\
&=\frac{1}{2}\sqrt{\left\|\begin{bmatrix}
2P_{\overline{\mathcal{R}(A)}}+ TT^{\#_A}& -TP_{\overline{\mathcal{R}(A)}}+P_{\overline{\mathcal{R}(A)}}T\\-T^{\#_A}+T^{\#_A}&2P_{\overline{\mathcal{R}(A)}}+T^{\#_A}T
\end{bmatrix}\right\|_\mathbb{A}+2}\\
&=\frac{1}{2}\sqrt{\left\|\begin{bmatrix}
2P_{\overline{\mathcal{R}(A)}}+ TT^{\#_A}& -TP_{\overline{\mathcal{R}(A)}}+P_{\overline{\mathcal{R}(A)}}T\\O&2P_{\overline{\mathcal{R}(A)}}+T^{\#_A}T
\end{bmatrix}\right\|_\mathbb{A}+2}\\
&=\frac{1}{2}\sqrt{\left\|\begin{bmatrix}
2P_{\overline{\mathcal{R}(A)}}+ (T^{\#_A})^{\#_A}T^{\#_A}& O \\ -P_{\overline{\mathcal{R}(A)}}T^{\#_A}+T^{\#_A}P_{\overline{\mathcal{R}(A)}} &2P_{\overline{\mathcal{R}(A)}}+T^{\#_A}(T^{\#_A})^{\#_A}
\end{bmatrix}\right\|_\mathbb{A}+2}\\&~~~~~~~~~~~~~~~~~~~~~~~~~~~~~~~~~~~~~~~~~~~~~~~~~~~~~\mbox{as}~ \|T\|_A=\|T^{\#_A}\|_A ~\mbox{and}~ (P_{\overline{\mathcal{R}(A)}})^{\#_A}=P_{\overline{\mathcal{R}(A)}}\\
\end{align*}
\begin{align*}
&=\frac{1}{2}\sqrt{\left\|\begin{bmatrix}
2P_{\overline{\mathcal{R}(A)}}+ (T^{\#_A})^{\#_A}T^{\#_A}& O  \\ -T^{\#_A}+T^{\#_A} &2P_{\overline{\mathcal{R}(A)}}+T^{\#_A}(T^{\#_A})^{\#_A}
\end{bmatrix}\right\|_\mathbb{A}+2}\\
&=\frac{1}{2}\sqrt{\left\|\begin{bmatrix}
2P_{\overline{\mathcal{R}(A)}}+ (T^{\#_A})^{\#_A}T^{\#_A}& O  \\ O &2P_{\overline{\mathcal{R}(A)}}+T^{\#_A}(T^{\#_A})^{\#_A}
\end{bmatrix}\right\|_\mathbb{A}+2}\\
&=\frac{1}{2}\sqrt{\left\|\begin{bmatrix}
2I^{\#_A}+ (T^{\#_A})^{\#_A}T^{\#_A}& O  \\ O &2I^{\#_A}+T^{\#_A}(T^{\#_A})^{\#_A}
\end{bmatrix}\right\|_\mathbb{A}+2}\\
&=\frac{1}{2}\sqrt{\left\|\begin{bmatrix}
2I+ TT^{\#_A}& O\\O&2I+T^{\#_A}T
\end{bmatrix}\right\|_\mathbb{A}+2}\\
&=\frac{1}{2}\max\{(\|2I+TT^{\#_A}\|_A+2)^{1/2},(\|2I+T^{\#_A}T\|_A+2)^{1/2}\}\\
&=\frac{1}{2}(\|2I+TT^{\#_A}\|_A+2)^{1/2}\\
&=\frac{1}{2}\sqrt{\|T\|_A^2+4}.
\end{align*}
So, we get
\begin{equation}\label{eqn_0003.15}
    w_{\mathbb{A}}\left(\begin{bmatrix}
I& T\\O&-I
\end{bmatrix}\right)=\frac{1}{2}\sqrt{\|T\|_A^2+4}.
\end{equation}
Using Theorem \ref{thm_00003.8},  we also obtain
\begin{equation}\label{eqn_0003.9}
\left\|\begin{bmatrix}
I& T\\O&-I
\end{bmatrix}\right\|_\mathbb{A}^2=\frac{1}{2}\left(2+\|T\|_A^2+\sqrt{\|T\|_A^4+4\|T\|_A^2}\right)=\frac{1}{2}\|T\|_A+\frac{1}{2}\sqrt{\|T\|_A^2+4}.
\end{equation}
Hence,   we arrive at our claim by \eqref{eqn_0003.15} and \eqref{eqn_0003.9}.
\end{proof}

Using Theorem \ref{thm_00003.8},  one can establish Corollary 2.2 \cite{Feki_some} without the assumption  ``$\mathcal{N}(A)^\perp$ is invariant under $T$.''  The same is stated next without the proof.
\begin{cor}
Let $T\in\mathcal{B}_A(\mathcal{H}).$ Then\\
\begin{enumerate}
    \item [\textnormal{(i)}] $\left\|Re_\mathbb{A}\left(\begin{bmatrix}I&T\\O&-I\end{bmatrix}\right)\right\|_\mathbb{A}=w_\mathbb{A}\left(\begin{bmatrix}I&T\\O&-I\end{bmatrix}\right).$\\
    
    \item [\textnormal{(ii)}] $\left\|Im_\mathbb{A}\left(\begin{bmatrix}I&T\\O&-I\end{bmatrix}\right)\right\|_\mathbb{A}=\dfrac{1}{2}\left(\left\|\begin{bmatrix}I&T\\O&-I\end{bmatrix}\right\|_\mathbb{A}-\left\|\begin{bmatrix}I&T\\O&-I\end{bmatrix}\right\|_\mathbb{A}^{-1}\right).$
\end{enumerate}

\end{cor}
The following lemma provides an upper bound for $T\in\mathcal{B}_A(\mathcal{H})$ to prove Theorem \ref{thm3.5}.
\begin{lem}[Theorem 7, \cite{Feki_spect}]\label{lem_03.4}
Let $T\in \mathcal{B}_A(\mathcal{H}).$ Then
$$w_A(T)\leq\frac{1}{2}(\|T\|+\|T^2\|^{1/2}).$$
\end{lem}

\begin{thm}\label{thm3.5}
Let $T_1,T_2,T_3,T_4\in \mathcal{B}_A(\mathcal{H})$ and $T=\begin{bmatrix}
T_1&T_2\\T_3&T_4
\end{bmatrix}.$ Then
$$\max\{w_A^{1/2}(T_2T_3),w_A^{1/2}(T_3T_2)\}\leq w_\mathbb{A}\left(\begin{bmatrix}
O&T_2\\T_3&O
\end{bmatrix}\right)\leq \frac{1}{2}(\|T\|_\mathbb{A}+\|T^2\|^{1/2}).$$
\end{thm}
\begin{proof}
Let $U=\begin{bmatrix}
 I&O\\O&-I
\end{bmatrix}.$ It is easy to see that $U$ is $A$-unitary and
  $TU-UT=2\begin{bmatrix}
O&-T_2\\T_3&O
\end{bmatrix}.$\\
Here, \begin{align}\label{eqn_003.3}
    w_\mathbb{A}(TU\pm UT)&=w_\mathbb{A}(U^{\#_A}T^{\#_A}\pm T^{\#_A}U^{\#_A})~~~~~~~~~~~~~~~\because~w_A(T)=w_A(T^{\#_A})\nonumber\\
    &=w_\mathbb{A}(U^{\#_A}T^{\#_A}\pm T^{\#_A}(U^{\#_A})^{\#_A})~~~~~~~~~~ \because  ~U^{\#_A}=(U^{\#_A})^{\#_A}\nonumber\\
    &\leq 2w_\mathbb{A}(T^{\#_A})\|U^{\#_A}\|_\mathbb{A}~~by~Lemma~\ref{lemm00001}\nonumber\\
    &=2w_\mathbb{A}(T)\nonumber\\
    &\leq \|T\|_\mathbb{A}+\|T^2\|^{1/2}~~~~by~Lemma~\ref{lem_03.4}.
\end{align}
Now,
\begin{align*}
    \max\{w_A(T_2T_3),w_A(T_3T_2)\}&=w_\mathbb{A}\left(\begin{bmatrix}
T_2T_3&O\\O&T_3T_2
\end{bmatrix}\right)\\
&=w_\mathbb{A}\left(\begin{bmatrix}
O&T_2\\T_3&O
\end{bmatrix}\begin{bmatrix}
O&T_2\\T_3&O
\end{bmatrix}\right)\\
% \end{align*}
% \begin{align*}
&=w_{\mathbb{A}}\left(\begin{bmatrix}
O&T_2\\T_3&O
\end{bmatrix}^2\right)\\
&\leq w_\mathbb{A}^2\left(\begin{bmatrix}
O&T_2\\T_3&O
\end{bmatrix}\right)~~~~~~~~~\because w(T^n)\leq w^n(T).
\end{align*}
Replacing  $T_2$ by $-T_2$, we get
$$\max\{w_A(T_2T_3),w_A(T_3T_2)\}\leq w_\mathbb{A}^2\left(\begin{bmatrix}
O&-T_2\\T_3&O
\end{bmatrix}\right).$$
This implies 
\begin{align*}
    \max\{w_A^{1/2}(T_2T_3),w_A^{1/2}(T_3T_2)\}&\leq w_\mathbb{A}\left(\begin{bmatrix}
O&-T_2\\T_3&O
\end{bmatrix}\right)\\
&=\frac{1}{2}w_\mathbb{A}(TU-UT)\\
&\leq \frac{1}{2}(\|T\|_\mathbb{A}+\|T^2\|^{1/2})~~by~~ \eqref{eqn_003.3}.
\end{align*}
By Lemma \ref{lemm_0000}, we thus obtain $$\max\{w_A^{1/2}(T_2T_3),w_A^{1/2}(T_3T_2)\}\leq w_\mathbb{A}\left(\begin{bmatrix}
O&T_2\\T_3&O
\end{bmatrix}\right)\leq \frac{1}{2}(\|T\|_\mathbb{A}+\|T^2\|^{1/2}).$$
\end{proof}

%%%%%%%%%%%%

We generalize some of the results of \cite{Hir_Kit_2011} now.
Using Lemma \ref{lemm_0000}, one can now prove Corollary 3.3 \cite{PINTU}  without assuming the condition $A>0$, and is stated next.

\begin{lem}\label{lem_00002}
Let $T,S,X,Y\in\mathcal{B}_A(\mathcal{H}).$ Then
$$w_A(TXS^{\#_A}\pm SYT^{\#_A})\leq 2\|T\|_A\|S\|_Aw_{\mathbb{A}}\begin{bmatrix}
O&X\\Y&O\end{bmatrix}.$$
In particular, putting $Y=X$
$$w_A(TXS^{\#_A}\pm SXT^{\#_A})\leq 2\|T\|_A\|S\|_Aw_A(X).$$
\end{lem}

Considering $X=Y=Q$ and $T=I$ in  Lemma \ref{lem_00002}, we get Lemma \ref{lemm00001}, which is stated below.  
\begin{cor}\label{cor_pin_modified}
Let $Q,S\in\mathcal{B}_A(\mathcal{H}).$ Then
$$w_A(QS^{\#_A}\pm SQ)\leq 2\|S\|_Aw_A(Q).$$
\end{cor}

It is well known  that  $P_{\overline{R(A)}}T\neq TP_{\overline{R(A)}}$ for  $T\in\mathcal{B}_A(\mathcal{H})$ (even if $A$ and $T$ are finite matrices). 
 And the equality holds if $\mathcal{N}(A)^\perp$ is invariant under $T.$ The following result shows that  $w_A(P_{\overline{R(A)}}T)$ and $w_A(TP_{\overline{R(A)}})$ are same  for any $T\in\mathcal{B}_A(\mathcal{H})$ even though $\mathcal{N}(A)^\perp$ is not invariant under $T.$

\begin{thm}
 $w_A(P_{\overline{\mathcal{R}(A)}}T)=w_A(TP_{\overline{\mathcal{R}(A)}})=w_A(T)$ for any $T\in\mathcal{B}_A(\mathcal{H}).$
\end{thm}
\begin{proof}
\begin{align}\label{eqn_100007}
 	    w_A(P_{\overline{R(A)}}T)&=w_A((P_{\overline{R(A)}}T)^{\#_A})~~~ ~~(\because~ w_A(T)=w_A(T^{\#_A}))\nonumber\\
 	    &=w_A(T^{\#_A}P_{\overline{\mathcal{R}(A)}})~~~~~~~~(\because~ (TS)^{\#_A}=S^{\#_A}T^{\#_A}~  \&~ (P_{\overline{\mathcal{R}(A)}})^{\#_A}=P_{\overline{\mathcal{R}(A)}} )\nonumber\\
 	    &=w_A(T^{\#_A})~~~~~~~~~~~~~~~~~by~\eqref{eqn00_1.8}\nonumber\\
 	    &=w_A(T).
 	\end{align}  
Again,
 	\begin{align}\label{eqn_100008}
 	     w_A(TP_{\overline{R(A)}})&=w_A((TP_{\overline{R(A)}})^{\#_A}) ~~~ ~~(\because~ w_A(T)=w_A(T^{\#_A}))\nonumber\\
 	    &=w_A(P_{\overline{\mathcal{R}(A)}}T^{\#_A})~~~~~~~~(\because~ (TS)^{\#_A}=S^{\#_A}T^{\#_A} ~\&~ ~(P_{\overline{\mathcal{R}(A)}})^{\#_A}=P_{\overline{\mathcal{R}(A)}} )\nonumber\\
 	    &=w_A(T^{\#_A})~~~~~~~~~~~~~~~~~by~\eqref{eqn00_1.9}\nonumber\\
 	    &=w_A(T).
 	\end{align}
 We therefore have
 \begin{equation*}
 	    w_A(P_{\overline{\mathcal{R(A)}}}T)=w_A(TP_{\overline{\mathcal{R(A)}}})=w_A(T).
 	\end{equation*}
\end{proof}

Feki and Sahoo \cite{FekSat} established many results on $A$-numerical radius inequalities of $2\times 2 $ operator matrices, very recently. In many cases, they assumed the condition ``$\mathcal{N}(A)^\perp$ is invariant under  $T_1,T_2,T_3,T_4$''  to show their claim. They assumed these conditions in order to get the equality $P_{\overline{R(A)}}T= TP_{\overline{R(A)}}$ which is not true, in general.    One of the objective of this paper is to achieve the same claim without assuming the additional condition  ``$\mathcal{N}(A)^\perp$ is invariant under $T_1,T_2,T_3,T_4\in\mathcal{B}_A(\mathcal{H})$''.  
The next result is in this direction, and is more general than Theorem 2.7 \cite{FekSat}.    Our proof is also completely different than the corresponding proof in \cite{FekSat}. And, therefore our results are superior to those results in \cite{FekSat} and  \cite{Feki_some} that assumes the invariant condition.

\begin{thm}
Let $T_1,T_2,T_3,T_4\in \mathcal{B}_A(\mathcal{H}).$ Then $w_\mathbb{A}\left(\begin{bmatrix}
T_1 &T_2\\ T_3& T_4
\end{bmatrix}\right)\geq \frac{1}{2}\max\{\alpha, \beta\},$
where $\alpha=\max\{ w_A(T_1+T_2+T_3+T_4),~w_A(T_1+T_4-T_2-T_3)\}$
and $\beta=\max\{w_A(T_1+T_4+i(T_2-T_3)),~w_A(T_1+T_4-i(T_2-T_4))\}.$
\end{thm}

\begin{proof}
Let $T=\begin{bmatrix}
T_1^{\#_A} & T_3^{\#_A}\\ T_2^{\#_A} & T_4^{\#_A}
\end{bmatrix}$ and $Q=\begin{bmatrix}
O& I\\I & O
\end{bmatrix}.$ 
 To show that $Q$ is $\mathbb{A}$-unitary, we need to prove that $\|x\|_\mathbb{A}=\|Qx\|_\mathbb{A}=\|Q^{\#_\mathbb{A}}x\|_\mathbb{A}.$ So,
\begin{align*}
    Q^{\#_\mathbb{A}}&=  \begin{bmatrix}
O & I^{\#_A}\\
I^{\#_A} & O
\end{bmatrix}  ~~\mbox{by Lemma \ref{lem_fek}} \\
&=\begin{bmatrix}
 O& P_{\overline{\mathcal{R}(A)}} \\
P_{\overline{\mathcal{R}(A)}} & O 
\end{bmatrix} ~~~\because ~~ N(A)^{\perp}=\overline{\mathcal{R}(A^*)} ~~\&~~ \mathcal{R}(A^*)=\mathcal{R}(A).
\end{align*}
This in turn implies
    $QQ^{\#_\mathbb{A}}=\begin{bmatrix}
P_{\overline{\mathcal{R}(A)}} & O\\
O & P_{\overline{\mathcal{R}(A)}}
\end{bmatrix}= Q^{\#_\mathbb{A}}Q$.
Now, for $x=(x_1,x_2)\in \mathcal{H}\bigoplus \mathcal{H}$, we have
\begin{align*}
    \|Qx\|_\mathbb{A}^2=\langle Qx, Qx\rangle_\mathbb{A}=\langle Q^{\#_\mathbb{A}}Qx,x\rangle_\mathbb{A} &=\left\langle \begin{bmatrix}
P_{\overline{\mathcal{R}(A)}} & O\\
O & P_{\overline{\mathcal{R}(A)}}
 \end{bmatrix}\begin{bmatrix}
x_1\\
x_2
 \end{bmatrix},\begin{bmatrix}
x_1\\
x_2
 \end{bmatrix}\right\rangle_\mathbb{A}\\
 &=\left\langle\begin{bmatrix}
AP_{\overline{\mathcal{R}(A)}} & O\\
O & AP_{\overline{\mathcal{R}(A)}}
 \end{bmatrix}\begin{bmatrix}
x_1\\
x_2
 \end{bmatrix},\begin{bmatrix}
x_1\\
x_2
 \end{bmatrix}\right\rangle\\
  & = \left\langle\begin{bmatrix}
 AA^\dagger A & O\\
O & AA^\dagger A
 \end{bmatrix}\begin{bmatrix}
x_1\\
x_2
 \end{bmatrix},\begin{bmatrix}
x_1\\
x_2
 \end{bmatrix}\right\rangle\\
 &=\left\langle\begin{bmatrix}
A & O\\
O & A
 \end{bmatrix}\begin{bmatrix}
x_1\\
x_2
 \end{bmatrix},\begin{bmatrix}
x_1\\
x_2
 \end{bmatrix}\right\rangle\\
 &=\|x\|_{\mathbb{A}}^2.
\end{align*}
 So, $\|Qx\|_\mathbb{A}=\|x\|_\mathbb{A}.$ Similarly, it can be proved that $\|Q^{\#_\mathbb{A}}x\|_\mathbb{A}=\|x\|_\mathbb{A}.$ Thus, $Q$ is an $\mathbb{A}$-unitary operator.
 By Lemma \ref{lemm00001},  we obtain
\begin{equation}
    w_A(TQ\pm QT^{\#_A})\leq 2w_A(T).
\end{equation}
So, \begin{align*}
    2w_{\mathbb{A}}(T)&\geq w_{\mathbb{A}}\left(\begin{bmatrix}
T_1^{\#_A} & T_3^{\#_A}\\ T_2^{\#_A} & T_4^{\#_A}
\end{bmatrix}\begin{bmatrix}
O& P_{\overline{\mathcal{R}(A)}}\\P_{\overline{\mathcal{R}(A)}} & O
\end{bmatrix}+\begin{bmatrix}
O& I\\I & O
\end{bmatrix}\begin{bmatrix}
T_1^{\#_A} & T_3^{\#_A}\\ T_2^{\#_A} & T_4^{\#_A}
\end{bmatrix}\right)\\
&=w_{\mathbb{A}}\left(\begin{bmatrix}
T_3^{\#_A}P_{\overline{\mathcal{R}(A)}} & T_1^{\#_A}P_{\overline{\mathcal{R}(A)}}\\ T_4^{\#_A}P_{\overline{\mathcal{R}(A)}} & T_2^{\#_A}P_{\overline{\mathcal{R}(A)}}
\end{bmatrix}+\begin{bmatrix}
T_2^{\#_A} & T_4^{\#_A}\\ T_1^{\#_A} & T_3^{\#_A}
\end{bmatrix}\right)\\
% \end{align*}
% \begin{align*}
&=w_{\mathbb{A}}\left(\begin{bmatrix}
T_3^{\#_A} & T_1^{\#_A}\\ T_4^{\#_A} & T_2^{\#_A}
\end{bmatrix}+\begin{bmatrix}
T_2^{\#_A} & T_4^{\#_A}\\ T_1^{\#_A} & T_3^{\#_A}
\end{bmatrix}\right)~~~~~\mbox{by \eqref{eqn00_1.8}}\\
\end{align*}
\begin{align*}
&=w_{\mathbb{A}}\left(\begin{bmatrix}
T_3^{\#_A}+T_2^{\#_A} & T_1^{\#_A}+T_4^{\#_A}\\ T_4^{\#_A}+T_1^{\#_A} & T_2^{\#_A}+T_3^{\#_A}
\end{bmatrix}\right)\\
&=w_{\mathbb{A}}\left(\begin{bmatrix}
T_2+T_3 & T_4+T_1\\ T_4+T_1 & T_2+T_3
\end{bmatrix}^{\#_{\mathbb{A}}}\right)=w_{\mathbb{A}}\left(\begin{bmatrix}
T_2+T_3 & T_4+T_1\\ T_4+T_1 & T_2+T_3
\end{bmatrix}\right).
\end{align*}
Hence, we have 
\begin{equation}\label{eqn00001}
    2w_{\mathbb{A}}\left(\begin{bmatrix}
T_1 &T_2\\ T_3 & T_4
\end{bmatrix}.\right)=2w_{\mathbb{A}}\left(\begin{bmatrix}
T_1^{\#_A} &T_3^{\#_A}\\ T_2^{\#_A} & T_4^{\#_A}
\end{bmatrix}.\right)\geq w_{\mathbb{A}}\left(\begin{bmatrix}
T_2+T_3 & T_4+T_1\\ T_4+T_1 & T_2+T_3
\end{bmatrix}\right).
\end{equation}
By \eqref{eqn00001} and Lemma \ref{lemm_0000}, we obtain 
\begin{equation}\label{eqn00002}
    w_\mathbb{A}\left(\begin{bmatrix} T_1 & T_2\\T_3& T_4
\end{bmatrix}\right)\geq \frac{1}{2}\max\{w_A(T_1+T_2+T_3+T_4), w_A(T_2+T_3-T_4-T_1)\}.
\end{equation}
Again, applying Lemma \ref{lemm00001} and taking  $T=\begin{bmatrix}
T_1^{\#_A} &T_3^{\#_A}\\ T_2^{\#_A} & T_4^{\#_A}
\end{bmatrix}$ and $Q=\begin{bmatrix}
O&I\\ -I & O
\end{bmatrix}.$ It is easy to verify that $Q$ is $\mathbb{A}$-unitary. We now have by Lemma \ref{lemm00001}
\begin{equation}
    w_\mathbb{A}\left(TQ^{\#_A}\pm QT\right)\leq 2w_\mathbb{A}(T).
\end{equation}
So, 
\begin{align*}
    2w_\mathbb{A}(T)&\geq w_{\mathbb{A}}\left(\begin{bmatrix}
T_1^{\#_A} & T_3^{\#_A}\\ T_2^{\#_A} & T_4^{\#_A}
\end{bmatrix}
\begin{bmatrix}
O& -P_{\overline{\mathcal{R}(A)}}\\P_{\overline{\mathcal{R}(A)}} & O
\end{bmatrix}-\begin{bmatrix}
O& I\\-I & O
\end{bmatrix}\begin{bmatrix}
T_1^{\#_A} & T_3^{\#_A}\\ T_2^{\#_A} & T_4^{\#_A}
\end{bmatrix}\right)\\
&=w_{\mathbb{A}}\left(\begin{bmatrix}
T_3^{\#_A}P_{\overline{\mathcal{R}(A)}} & -T_1^{\#_A}P_{\overline{\mathcal{R}(A)}}\\ T_4^{\#_A}P_{\overline{\mathcal{R}(A)}} & -T_2^{\#_A}P_{\overline{\mathcal{R}(A)}}
\end{bmatrix}-\begin{bmatrix}
T_2^{\#_A} & T_4^{\#_A}\\ -T_1^{\#_A} & -T_3^{\#_A}
\end{bmatrix}\right)\\
&=w_{\mathbb{A}}\left(\begin{bmatrix}
-T_2^{\#_A}+T_3^{\#_A} & -T_4^{\#_A}-T_1^{\#_A}\\ T_4^{\#_A}+T_1^{\#_A} & -T_2^{\#_A}+T_3^{\#_A}
\end{bmatrix}\right)~~~\mbox{by \eqref{eqn00_1.8}}\\
&=w_{\mathbb{A}}\left(\begin{bmatrix}
-T_2+T_3 & T_4+T_1\\ -T_4-T_1 & -T_2+T_3
\end{bmatrix}\right).
\end{align*}
By  Lemma \ref{l002},  we therefore achieve the following:
\begin{equation}\label{eqn00003}
    w_\mathbb{A}\left(\begin{bmatrix}
T_1&T_2\\T_3&T_4
\end{bmatrix}\right)\geq \frac{1}{2}\max\{w_A(T_4+T_1-i(T_2-T_3)), w_A(T_4+T_1+i(T_2-T_3))\}.
\end{equation}
From \eqref{eqn00002} and \eqref{eqn00003}, we get the desired result.
\end{proof}
We provide below the same estimate as in Theorem 2.8 \cite{FekSat} for $\mathbb{A}$-numerical radius of an operator matrix that improves but by dropping the assumption $\mathcal{N}(A)^\perp$ is  invariant under $T_1,T_2\in\mathcal{B}_A(\mathcal{H}).$

\begin{thm}
Let $T_1,T_2\in\mathcal{B}_A(\mathcal{H}).$ Then $$w_\mathbb{A}\left(\begin{bmatrix}
T_1 & T_2\\O& O
\end{bmatrix}\right)\geq \frac{1}{2}\max \{w_A(T_1+iT_2),w_A(T_1-iT_2)\}.$$
\end{thm}
\begin{proof}
Suppose that  $T=\begin{bmatrix}
T_1^{\#_A} & O\\T_2^{\#_A}& O
\end{bmatrix}$ and $Q=\begin{bmatrix}
 O & -I\\I & O
\end{bmatrix}.$ It then follows  that $Q$ is $\mathbb{A}$-unitary. So, $\|Q\|_\mathbb{A}=1.$ Using Lemma \ref{lemm00001}, we get
$$2w_\mathbb{A}(T)\geq w_\mathbb{A}(TQ^{\#_A}-QT).$$
 Now,
 \begin{align*}
w_\mathbb{A}(T)&\geq \frac{1}{2}w_\mathbb{A}(TQ^{\#_\mathbb{A}}-QT)\\
&=\frac{1}{2}w_\mathbb{A}\left(\begin{bmatrix}
T_1^{\#_A} & O\\T_2^{\#_A}& O
\end{bmatrix}\begin{bmatrix}
 O & P_{\overline{\mathcal{R}(A)}}\\-P_{\overline{\mathcal{R}(A)}}& O
\end{bmatrix}-\begin{bmatrix}
 O& -I\\I& O
\end{bmatrix}\begin{bmatrix}
T_1^{\#_A} & O\\T_2^{\#_A}& O
\end{bmatrix}\right)\\
&=\frac{1}{2}w_\mathbb{A}\left(\begin{bmatrix}
O& T_1^{\#_A}P_{\overline{\mathcal{R}(A)}}\\O& T_2^{\#_A}P_{\overline{\mathcal{R}(A)}}
\end{bmatrix}-\begin{bmatrix}
  -T_2^{\#_A}&O\\T_1^{\#_A}& O
\end{bmatrix}\right)\\
% \end{align*}
%   \begin{align*}
&=\frac{1}{2}w_{\mathbb{A}}\left(\begin{bmatrix}
T_2^{\#_A} & T_1^{\#_A}\\-T_1^{\#_A}& T_2^{\#_A}
\end{bmatrix}\right)~~by~\eqref{eqn00_1.8}\\
&=\frac{1}{2}w_{\mathbb{A}}\left(\begin{bmatrix}
T_2 & -T_1\\T_1& T_2
\end{bmatrix}^{\#_\mathbb{A}}\right)\\
&=\frac{1}{2}w_\mathbb{A}\left(\begin{bmatrix}
T_2 & -T_1\\T_1& T_2
\end{bmatrix}\right).
 \end{align*}
 By  Lemma \ref{l002}, we thus have 
 $$w_\mathbb{A}\left(\begin{bmatrix}
T_1 & T_2\\O & O
\end{bmatrix}\right)=w_\mathbb{A}\left(\begin{bmatrix}
T_1^{\#_A} & O\\T_2^{\#_A}& O
\end{bmatrix}\right)\geq \frac{1}{2}\max\{w_\mathbb{A}(T_1+iT_2),w_\mathbb{A}(T_1-iT_2)\}.$$
\end{proof}
\begin{cor}
Let $T=P+iQ$ be the cartesian decomposition in $\mathcal{B}_A(\mathcal{H})$.  Then
$$\frac{1}{2}w_A(T)\leq \min\left\{w_\mathbb{A}\left(\begin{bmatrix}
P & Q\\O & O
\end{bmatrix}\right), w_\mathbb{A}\left(\begin{bmatrix}
O & P\\Q & O
\end{bmatrix}\right)\right\}.$$
\end{cor}
\begin{proof}
\begin{align}\label{eqn00004}
    w_\mathbb{A}\left(\begin{bmatrix}
P & Q\\O & O
\end{bmatrix}\right)&\geq \frac{1}{2}\max\{w_A(P+iQ),w_A(P-iQ)\}\nonumber\\
&=\frac{1}{2}\max\{w_A(T),w_A(T^{\#_A})\}\nonumber\\
&=\frac{1}{2}w_\mathbb{A}(T).
\end{align}
% Again, replacing $T_2$ and $T_3$ by $P$ and $iQ$, respectively in Lemma 2.12 and 
Using Lemma \ref{lemm_0000}, we obtain

\begin{equation}\label{eqn00005}
    w_\mathbb{A}\left(\begin{bmatrix} O&P\\Q&O
\end{bmatrix}\right)=w_\mathbb{A}\left(\begin{bmatrix} O&P\\iQ&O
\end{bmatrix}\right)\geq \frac{1}{2}w_A(P\pm iQ)=\frac{1}{2}w_A(T).
\end{equation}
From \eqref{eqn00004} and \eqref{eqn00005}, we have $$\frac{1}{2}w_A(T)\leq \min\left\{w_\mathbb{A}\left(\begin{bmatrix}
P & Q\\O & O
\end{bmatrix}\right), w_\mathbb{A}\left(\begin{bmatrix}
O & P\\Q & O
\end{bmatrix}\right)\right\}.$$
\end{proof}

We remark that the condition ``$\mathcal{N}(A)^\perp$ is invariant under operators in $\mathcal{B}_A(\mathcal{H})$'' in Theorem 2.9 \cite{FekSat} can also be dropped, similarly and is presented next.
Next, we recall a lemma that is used to prove Theorem \ref{t3.4}.

\begin{lem}\label{lem0003.1} [Lemma 2.6, \cite{Feki}]\\
Let $X,Y\in\mathcal{B}_A(\mathcal{H}).$ Then 
$$w_\mathbb{A}\left(\begin{bmatrix}
O&X\\Y&O
\end{bmatrix}\right)=\frac{1}{2}\sup_{\theta\in\mathbb{R}}\|e^{i\theta} X+e^{-i\theta}Y^{\#_A}\|_A.$$
\end{lem}

\begin{thm}\label{t3.4}
Let $T_1,T_2\in\mathcal{B}_A(\mathcal{H}).$ Then 
$$w_{\mathbb{A}}^4\left(\begin{bmatrix}
O&T_1\\T_2&O
\end{bmatrix}\right)\leq \frac{1}{16}\|P\|^2+\frac{1}{4}w^2_A(T_2T_1)+\frac{1}{8}w_A(PT_2T_1+T_2T_1P)$$
where $P=T_1^{\#_A}T_1+T_2T_2^{\#_A}.$
\end{thm}

\begin{proof}
Let $T=\begin{bmatrix}
O&T_1\\T_2&O
\end{bmatrix},$  $P=T_1^{\#_A}T_1+T_2T_2^{\#_A}$ and $\theta\in\mathbb{R}.$
Now,
\begin{align*}
    &\frac{1}{2}\|e^{i\theta} T_1+e^{-i\theta}T_2^{\#_A}\|_A\\
   &=\frac{1}{2}\|(e^{i\theta} T_1+e^{-i\theta}T_2^{\#_A})^{\#_A}(e^{i\theta} T_1+e^{-i\theta}T_2^{\#_A})\|_A^\frac{1}{2}\\
     \end{align*}
\begin{align*}
    &=\frac{1}{2}\|(e^{-i\theta} T_1^{\#_A}+e^{i\theta}(T_2^{\#_A})^{\#_A})(e^{i\theta} T_1+e^{-i\theta}T_2^{\#_A})\|_A^\frac{1}{2}\\
    &=\frac{1}{2}\|T_1^{\#_A}T_1+e^{-2i\theta}T_1^{\#_A}T_2^{\#_A}+e^{2i\theta}(T_2^{\#_A})^{\#_A}T_1+(T_2^{\#_A})^{\#_A}T_2^{\#_A}\|_A^{\frac{1}{2}}\\
    &=\frac{1}{2}\|T_1^{\#_A}(T_1^{\#_A})^{\#_A}+e^{2i\theta}(T_2^{\#_A})^{\#_A}(T_1^{\#_A})^{\#_A}+e^{-2i\theta}T_1^{\#_A}T_2^{\#_A}+(T_2^{\#_A})^{\#_A}T_2^{\#_A}\|_A^{\frac{1}{2}}~~(\because ~\|T\|_A=\|T^{\#_A}\|_A)\\
    &=\frac{1}{2}\|T_1^{\#_A}T_1+e^{-2i\theta}T_1^{\#_A}T_2^{\#_A}+e^{2i\theta}T_2T_1+T_2T_2^{\#_A}\|_A^{\frac{1}{2}}\\
   &=\frac{1}{2}\|T_1^{\#_A}T_1+T_2T_2^{\#_A}+(e^{2i\theta}T_2T_1)^{\#_A}+e^{2i\theta}T_2T_1\|_A^{\frac{1}{2}}\\
    &=\frac{1}{2}\|T_1^{\#_A}T_1+T_2T_2^{\#_A}+2Re(e^{2i\theta}T_2T_1)\|_A^{\frac{1}{2}}\\
    &=\frac{1}{2}\|(T_1^{\#_A}T_1+T_2T_2^{\#_A}+2Re(e^{2i\theta}T_2T_1))^2\|_A^{\frac{1}{4}}\\
%      \end{align*}
%  \begin{align*}
    &=\frac{1}{2}\|P^2+4(Re(e^{2i\theta}T_2T_1))^2+2P Re(e^{2i\theta}T_2T_1)+2Re(e^{2i\theta}T_2T_1)P\|_A^{\frac{1}{4}}\\
    &=\frac{1}{2}\|P^2+4(Re(e^{2i\theta}T_2T_1))^2+2 Re(e^{2i\theta}(PT_2T_1+T_2T_1P))\|_A^{\frac{1}{4}}.
\end{align*}
So, $$\left(\frac{1}{2}\|e^{i\theta} T_1+e^{-i\theta}T_2^{\#_A}\|_A\right)^4=\frac{1}{16}\|P^2+4(Re(e^{2i\theta}T_2T_1))^2+2 Re(e^{2i\theta}(PT_2T_1+T_2T_1P))\|_A.$$
This implies 
$$\left(\frac{1}{2}\|e^{i\theta} T_1+e^{-i\theta}T_2^{\#_A}\|_A\right)^4\leq\frac{1}{16}\|P\|_A^2+\frac{1}{4}\|Re_A(e^{2i\theta}T_2T_1)\|_A^2+\frac{1}{8} \|Re(e^{2i\theta}(PT_2T_1+T_2T_1P))\|_A.$$
Now, taking supremum over $\theta\in\mathbb{R}$ and using Lemma \ref{lem0003.1}, we thus obtain
$$w_{\mathbb{A}}^4\left(\begin{bmatrix}
O&T_1\\T_2&O
\end{bmatrix}\right)\leq \frac{1}{16}\|P\|^2+\frac{1}{4}w^2_A(T_2T_1)+\frac{1}{8}w_A(PT_2T_1+T_2T_1P).$$
\end{proof}

Note that the authors of \cite{PINTU1} proved the above theorem with the assumption  $A>0.$
Using  Theorem \ref{t3.4} and Lemma \ref{lemm_0000},  we now establish the following inequality.
\begin{cor}
Let $T_1,T_2\in\mathcal{B}_A(\mathcal{H}).$ Then $$w_A(T_1T_2)\leq \frac{1}{4}\sqrt{\|P\|^2+4w_A^2(T_2T_1)+2w_A(T_2T_1P+PT_2T_1)}$$ where $P=T_1^{\#_A}T_1+T_2T_2^{\#_A}.$
\end{cor}
\begin{proof}
Here\begin{align*}
    w_A(T_1T_2)&\leq\max\{w_A(T_1T_2),w_A(T_2T_1)\}\\
    &=w_\mathbb{A}\left(\begin{bmatrix}
T_1T_2&O\\
O&T_2T_1
\end{bmatrix}\right)\\
&=w_\mathbb{A}\left(\begin{bmatrix}
O&T_1\\
T_2&O
\end{bmatrix}^2\right)\\
&\leq w_\mathbb{A}^2\left(\begin{bmatrix}
O&T_1\\
T_2&O
\end{bmatrix}\right)\\
&\leq \frac{1}{4}\sqrt{\|P\|^2+4w_A^2(T_2T_1)+2w_A(T_2T_1P+PT_2T_1)}.
\end{align*}
The last inequality follows by Theorem \ref{t3.4}.
\end{proof}
Adopting a parallel technique as in the proof of the  Theorem \ref{t3.4}, one can prove the following result.   
\begin{thm}\label{thm3.8}
Let $T_1,T_2\in\mathcal{B}_A(\mathcal{H})$,  
\begin{equation}
    w_{\mathbb{A}}^4\left(\begin{bmatrix}
O&T_1\\T_2&O
\end{bmatrix}\right)\geq \frac{1}{16}\|P\|^2+\frac{1}{8}c_A(PT_2T_1+T_2T_1P)+\frac{1}{4}m^2_A(T_2T_1),
\end{equation}
where $P=T_1^{\#_A}T_1+T_2T_2^{\#_A}$ and  $m_A(T_2T_1)=\displaystyle \inf_{\theta\in\mathbb{R}}\inf_{\substack{x\in\mathcal{H}\\\|x\|_A=1}}\|Re(e^{i\theta}T_2T_1)x\|_A.$ 
\end{thm}
 \begin{proof}
 Let $x\in\mathcal{H}$ such that $\|x\|_A=1$ and $\theta\in\mathbb{R}.$ By Lemma \ref{lem0003.1} we have
 $$w_\mathbb{A}\left(\begin{bmatrix}
O&T_1\\T_2&O
\end{bmatrix}\right)\geq \frac{1}{2}\|e^{i\theta}T_1+e^{-i\theta}T_2\|_A.$$
Now using similar arguments as in Theorem \ref{t3.4} we can obtain,
\begin{align*}
    w_\mathbb{A}\left(\begin{bmatrix}
O&T_1\\T_2&O
\end{bmatrix}\right)&\geq \frac{1}{2}\|(T_1^{\#_A}T_1+T_2T_2^{\#_A})^2+4(Re(e^{2i\theta}T_2T_1))^2+2 Re(e^{2i\theta}(PT_2T_1+T_2T_1P))\|_A^{\frac{1}{4}}\\
&\geq \frac{1}{2}\left\langle\left((T_1^{\#_A}T_1+T_2T_2^{\#_A})^2+4(Re(e^{2i\theta}T_2T_1))^2+2 Re(e^{2i\theta}(PT_2T_1+T_2T_1P))\right)x,x\right\rangle_A|^{\frac{1}{4}}\\
&\geq \frac{1}{2}\left[\|(T_1^{\#_A}T_1+T_2T_2^{\#_A})x\|^2+4\|(T_2T_1)x\|^2+2| \langle (PT_2T_1+T_2T_1P)x,x\rangle_A|\right]^{\frac{1}{4}}.
\end{align*}
Now taking infimum over $x\in\mathcal{H}$ and $\theta\in\mathbb{R}$, we obtain the result.

 \end{proof}
The next result provides  upper and lower bounds for $A$-numerical radius of $2\times 2$ operator matrix which follows directly using Theorem \ref{t3.4}, Theorem \ref{thm3.8}  and Lemma \ref{l001}.

\begin{thm}
Let $T_1,T_2,T_3,T_4\in\mathcal{B}_A(\mathcal{H}).$ Then
$$w_{\mathbb{A}}\left(\begin{bmatrix}
T_1&T_2\\T_3&T_4
\end{bmatrix}\right)\leq \max\{w_A(T_1,w_A(T_4))\}+ [\frac{1}{16}\|P\|^2+\frac{1}{8}w_A(PT_3T_2+T_3T_2P)+\frac{1}{4}w_A^2(T_3T_2)]^{1/4},$$
and 
$$w_{\mathbb{A}}\left(\begin{bmatrix}
T_1&T_2\\T_3&T_4
\end{bmatrix}\right)\geq \max\{w_A(T_1,w_A(T_4)), [\frac{1}{16}\|P\|^2+\frac{1}{8}c_A(PT_3T_2+T_3T_2P)+\frac{1}{4}m_A^2(T_3T_2)]^{1/4}\},$$
where $P=T_1^{\#_A}T_1+T_2T_2^{\#_A}$ and $m_A(T_2T_1)=\displaystyle \inf_{\theta\in\mathbb{R}}\inf_{\substack{x\in\mathcal{H}\\\|x\|_A=1}}\|Re(e^{i\theta}T_2T_1)x\|_A.$ 
\end{thm}
%%%%%%%%%
We conclude this section with the following two results for $n\times n$ operator matrices.
First, we demonstrate an interesting property of $A-$numerical radius of an $n\times n$ operator matrix which is a generalization of Lemma 2.1 \cite{MEO}. 
\begin{thm}
Let $T=\begin{bmatrix}
 T_{11} & T_{12}  & \cdots  &   T_{1n}\\
 T_{21}  & T_{22} & \cdots  & T_{2n}\\
 \vdots & \vdots &\vdots & \vdots \\
T_{n1}  & T_{n2}  &\cdots&  T_{nn}
\end{bmatrix},$ where $T_{ij}\in\mathcal{B}_A(\mathcal{H})$ for $1\leq i,j\leq n.$ Then $$w_{\mathbb{A}}\left(\begin{bmatrix}
 T_{11} & O  & \cdots  &   O\\
 O  & T_{22} & \cdots  & O\\
 \vdots & \vdots &\vdots & \vdots \\
O  & O  &\cdots&  T_{nn}
\end{bmatrix}\right)\leq w_{\mathbb{A}}(T).$$
\end{thm}
\begin{proof}
Let $z=e^{\frac{2\pi i}{n}}$ and $U=\begin{bmatrix}
 I & O  & \cdots  &   O\\
 O  & zI & \cdots  & O\\
 \vdots & \vdots &\vdots & \vdots \\
O  & O  &\cdots&  z^{n-1}I
\end{bmatrix}.$ It is easy to see that $\overline{z}=z^{-1}=z^{n-1}$ and $|z|=1.$  
To show that $U$ is $\mathbb{A}$-unitary, we need to prove that $\|x\|_\mathbb{A}=\|Ux\|_\mathbb{A}=\|U^{\#_\mathbb{A}}x\|_\mathbb{A},$ for $x=(x_1,x_2,\cdots,x_n)\in \bigoplus_{i=1}^n \mathcal{H}.$ Here,
\begin{align*}
    U^{\#_\mathbb{A}}&=  \begin{bmatrix}
 I & O  & \cdots  &   O\\
 O  & zI & \cdots  & O\\
 \vdots & \vdots &\vdots & \vdots \\
O  & O  &\cdots&  z^{n-1}I
\end{bmatrix}^{\#_A}\\
\end{align*}
\begin{align*}
&=\begin{bmatrix}
 I^{\#_A} & O  & \cdots  &   O\\
 O  & \overline{z}I^{\#_A} & \cdots  & O\\
 \vdots & \vdots &\vdots & \vdots \\
O  & O  &\cdots&  \overline{z}^{n-1}I^{\#_A}
\end{bmatrix}~~ \mbox{by Lemma \ref{lem_fek}}\\
&=\begin{bmatrix}
 P_{\overline{\mathcal{R}(A)}} & O  & \cdots  &   O\\
 O  & \overline{z}P_{\overline{\mathcal{R}(A)}} & \cdots  & O\\
 \vdots & \vdots &\vdots & \vdots \\
O  & O  &\cdots&  \overline{z}^{n-1}P_{\overline{\mathcal{R}(A)}}
\end{bmatrix}.
\end{align*}
This in turn implies
    $UU^{\#_\mathbb{A}}=\begin{bmatrix}
 P_{\overline{\mathcal{R}(A)}} & O  & \cdots  &   O\\
 O  & P_{\overline{\mathcal{R}(A)}} & \cdots  & O\\
 \vdots & \vdots &\vdots & \vdots \\
O  & O  &\cdots&  P_{\overline{\mathcal{R}(A)}}
\end{bmatrix}=U^{\#_A}U.$\\
Now, for $x=(x_1,x_2,\cdots,x_n)\in \bigoplus_{i=1}^n \mathcal{H}$, we have
\begin{align*}
    \|Ux\|_\mathbb{A}^2=\langle Ux, Ux\rangle_\mathbb{A}=\langle U^{\#_\mathbb{A}}Ux,x\rangle_\mathbb{A}
 &=\|x\|_{\mathbb{A}}^2.
\end{align*}
 So, $\|Ux\|_\mathbb{A}=\|x\|_\mathbb{A}.$ Similarly,  $\|U^{\#_\mathbb{A}}x\|_\mathbb{A}=\|x\|_\mathbb{A}.$ Thus, $U$ is an $\mathbb{A}$-unitary operator.
Further, a simple calculation shows that 
$$\begin{bmatrix}
 T_{11}^{\#_A} & O  & \cdots  &   O\\
 O  & T_{22}^{\#_A} & \cdots  & O\\
 \vdots & \vdots &\vdots & \vdots \\
O  & O  &\cdots&  T_{nn}^{\#_A}
\end{bmatrix}=\frac{1}{n}\sum_{k=0}^{n-1} {U^{\#_{\mathbb{A}}}}^k T^{\#_{\mathbb{A}}}U^k.$$
So,
\begin{align*}
    w_{\mathbb{A}}\left(\begin{bmatrix}
 T_{11}^{\#_A} & O  & \cdots  &   O\\
 O  & T_{22}^{\#_A} & \cdots  & O\\
 \vdots & \vdots &\vdots & \vdots \\
O  & O  &\cdots&  T_{nn}^{\#_A}
\end{bmatrix}\right)&\leq \frac{1}{n}\sum_{k=0}^{n-1}w_{\mathbb{A}}({U^{\#_{\mathbb{A}}}}^k T^{\#_{\mathbb{A}}}U^k)\\
&=\frac{1}{n}\sum_{k=0}^{n-1}w_{\mathbb{A}}(T^{\#_{\mathbb{A}}})\\
&=\frac{1}{n}\sum_{k=0}^{n-1}w_{\mathbb{A}}(T)\\
&=w_{\mathbb{A}}(T).
\end{align*}
This implies that $$w_{\mathbb{A}}\left(\begin{bmatrix}
 T_{11} & O  & \cdots  &   O\\
 O  & T_{22} & \cdots  & O\\
 \vdots & \vdots &\vdots & \vdots \\
O  & O  &\cdots&  T_{nn}
\end{bmatrix}^{\#_\mathbb{A}}\right)=w_{\mathbb{A}}\left(\begin{bmatrix}
 T_{11} & O  & \cdots  &   O\\
 O  & T_{22} & \cdots  & O\\
 \vdots & \vdots &\vdots & \vdots \\
O  & O  &\cdots&  T_{nn}
\end{bmatrix}\right)\leq w_A(T).$$
\end{proof}

The next theorem provides a relation between  $A$-numerical radius of two diagonal operator matrices, where $diag(T_1,\ldots,T_n)$ means an $n\times n$ diagonal operator  matrix with entries $T_1,\ldots,T_n$.
\begin{thm}
Let $T_i\in\mathcal{B}_A(\mathcal{H})$ for $1\leq i\leq n.$ Then $$w_{\mathbb{A}}(diag(\sum_{i=1}^n T_i,\ldots, \sum_{i=1}^n T_i))\leq n w_\mathbb{A}(diag(T_1,\ldots,T_n)).$$
\end{thm}
\begin{proof}
Here,
\begin{align*}
   w_{\mathbb{A}}(diag(\sum_{i=1}^n T_i,\ldots, \sum_{i=1}^n T_i))&=w_A(\sum_{i=1}^n T_i)~~\mbox{by Lemma \ref{l001}}\\
   & \leq\sum_{i=1}^n w_A (T_i)\\
   &\leq n \max\{w_A(T_i): 1\leq i\leq n\}\\
   &=nw_\mathbb{A}(diag(T_1,\ldots,T_n)).
\end{align*}
\end{proof}

This paper ends with the note that further work on $A$-numerical radius for $n\times n $ operator matrices can be studied. \\

\vspace{.2cm}
\noindent
{\small {\bf Acknowledgments.}\\
We  thank the {\bf Government of India}  for introducing the {\it work from home initiative} during the COVID-19 pandemic.
}
\newpage

\bibliographystyle{amsplain}

\begin{thebibliography}{10}


%%%%%%%%%%%%%%%%%%%%%%%%%%%%%%%%%%%%%%%%%%%%

\bibitem{ARIS} Arias, M. L.; Corach, G.; Gonzalez, M. C., \textit{Metric properties of projections in semi-Hilbertian spaces}, Integr. Equ. Oper. Theory {\bf 62} (2008), 11--28.

 \bibitem{ARIS2} Arias, M. L.; Corach, G.; Gonzalez, M. C., \textit{Partial isometries in semi-Hilbertian spaces}, Linear Algebra Appl. {\bf 428} (2008), 1460--1475.


\bibitem{pinfek} Bhunia, P.; Feki, K.; Paul, K., \textit{$A$-numerical radius orthogonality and parallelism of semi-Hilbertian space operators and their applications,} Bull. Iran. Math. Soc. (2020), DOI: 10.1007/s41980-020-00392-8.


\bibitem{PINTU} Bhunia, P.; Paul, K.; Nayak, R. K., \textit{On inequalities for A-numerical radius of operators}, Electron. J. Linear Algebra \textbf{36} (2020), 143--157.


\bibitem{PINTU1} Bhunia, P.; Paul, K., \textit{Some improvements of numerical radius inequalities of operators and operator matrices,}  Linear Multilinear Algebra (2020), DOI: 10.1080/03081087.2020.1781037

\bibitem{Doug} Douglas, R. G., \textit{On majorization, factorization, and range inclusion of operators on Hilbert space},
Proc. Amer. Math. Soc. \textbf{17} (1966), 413--415.

\bibitem{Feki_spect} Feki, K., \textit{Spectral radius of semi-Hilbertian space operators and its applications,} Ann. Funct. Anal. (2020), DOI: 10.1007/s43034-020-00064-y.

\bibitem{Feki_some} Feki, K., \textit{Some numerical radius inequalities for semi-Hilbertian space operators}, arxiv:2001.00398v2[math.FA] (2020).

\bibitem{FekSat} Feki, K.; Sahoo, S., \textit{Further inequalities for the $\mathbb{A}$-numerical radius of
certain $2 \times 2$ operator matrices},	arXiv:2006.09312 [math.FA] (2020).

\bibitem{Feki} Feki, K., \textit{Some A-numerical radius inequalities for $d\times d$ operator matrices,} arXiv:2003.14378 [math.FA] (2020).

\bibitem{Feki1} Feki, K., \textit{Some $A$-spectral radius inequalities for A-bounded Hilbert space operators}, arXiv:2002.02905 [math.FA] (2020).

\bibitem{Hir_Kit_2011} Hirzallah. O; Kittaneh, F.; Shebrawi, K., \textit{Numerical Radius Inequalities for commutators of Hilbert space operators}, Numer. Funct. Anal. Optim.  {\bf 32} (2011), 739--749.

 \bibitem{Mos} Moslehian, M. S.; Kian, M.; Xu, Q.,  \textit{Positivity of $2\times 2$ block matrices of operators,} Banach J. Math. Anal. \textbf{13} (2019), 726--743.

 \bibitem{MOS} Moslehian, M. S.; Xu, Q.; Zamani, A.,  \textit{Seminorm and numerical radius inequalities of operators in semi-Hilbertian spaces}, Linear Algebra Appl. \textbf{591} (2020), 299--321.

 \bibitem{Nashed} Nashed, M. Z., \textit{Generalized Inverses and Applications}, Academic Press, New York, 1976.

\bibitem{Nirmal2} Rout, N. C.; Sahoo, S.; Mishra, D.,
\textit{On $\mathbb{A}$-numerical radius inequalities for $2\times2$ operator matrices}, under revision in Linear  Multilinear Algebra  (2020).  

 \bibitem{NSD} Rout, N. C.; Sahoo, S.; Mishra, D., \textit{Some $A$-numerical radius inequalities for semi-Hilbertian space operators}, Linear  Multilinear Algebra  (2020) DOI: 10.1080/03081087.2020.1774487
 
 
\bibitem{MEO} Omidvar, M. E.; Moslehian, M. S.; Niknam, A., \textit{Some numerical radius inequalities for Hilbert space operators}, Involve \textbf{2} (2009), 471-478. 

 \bibitem{Saddi} Saddi, A., \textit{A-normal operators in semi Hilbertian spaces,}  AJMAA \textbf{9} (2012), 1--12. 
 
 \bibitem{Zam} Zamani, A.,  \textit{A-Numerical radius inequalities for semi-Hilbertian space operators,} Linear Algebra Appl. \textbf{578} (2019), 159--183.
 
 


\end{thebibliography}

\end{document}